\documentclass[10pt]{amsart}
\usepackage{amsmath,amsthm,amssymb}
\usepackage{graphicx}

\newtheorem{theorem}{Theorem}

\newtheorem{criteria}{Criterion}
\newtheorem{corollary}{Corollary}

\theoremstyle{definition}

\newtheorem{definition}{Definition}

\newtheorem{remark}{Remark}

\newcommand{\R}{{\mathbb R}}
\newcommand{\gT}{{\mathfrak T}}
\newcommand{\const}{{\rm const}}
\newcommand{\essup}{{\rm essup}}
\newcommand{\sh}{{\rm sinh}}
\newcommand{\ch}{{\rm cosh}}

\begin{document}

\title[Disconjugacy of second order linear differential equation]{Disconjugacy of a second order linear differential equation and periodic solutions}

\author{V.~Ya.~Derr}

\begin{abstract}
The present paper is devoted to a new criterion for disconjugacy of a second order linear differential equation. Unlike most of the classical sufficient conditions for disconjugacy, our criterion does not involve assumptions on the smallness of the coefficients of the equation. We compare our criterion with several known criteria for disconjugacy, for which we provide detailed proofs, and discuss the 
applications of the property of disconjugacy to the 
problem of factorization of linear ordinary differential operators, and to the proof of the generalized Rolle's theorem. 
\end{abstract}

\email{vandv@udm.net}



\maketitle

\section{Introduction}

A linear differential equation 
\begin{equation}
\label{eq1}
(Lx)(t):= x''+p(t)x'+q(t)x=0,\quad I:=(a,b) \subset \mathbb R,
\end{equation}
having locally integrable coefficients $p$, $q:I \mapsto \mathbb R$,
is called {\it disconjugate} on an interval $J \subset I$ (open or closed) if any of its solutions $x \not\equiv 0$ can not have two zeros in $J$.
The property of disconjugacy became a subject of intense study in early 1950s (see, e.g.,  \cite{pol24} --- \cite{muld78}), in particular, due to the exceptional role that it plays
in the qualitative theory of second order linear differential equations. 
Traditionally (see, e.g., the literature cited above), 
most of the sufficient conditions for disconjugacy, formulated for 
differential equation (\ref{eq1}) written in the form $x''+Q(t)x=0$ (or $(P(t)x')'+Q(t)x=0$), include some kind of `smallness' assumption on the coefficient $Q$. 
In the present paper 
(which also may serve as a brief introduction to the theory of 
disconjugacy for second order linear differential equations) 
we obtain a new sufficient condition for disconjugacy for a 
differential equation of the general form (\ref{eq1}), that does not involve any assumptions on the smallness of the coefficients.

The paper is organized as follows. In Sections 2-4 we formulate and prove several 
known criteria for disconjugacy, and discuss the 
applications of the property of disconjugacy to the 
problem of factorization of linear ordinary differential operators, and to the proof of the generalized Rolle's theorem.
Section 5 is devoted to our new criterion for disconjugacy.

\section{Prelimiaries}

\subsection{Cauchy's and Green's functions}

\begin{definition}
A function $C:I\times [\alpha,t]\to \R$ is called \textit{Cauchy's function} of equation (\ref{eq1}) if
$$
(LC)(\cdot ,s)=0\quad \text{for almost all}\;t\geqslant s, \quad C(s,s)=0,\quad \frac{\partial C(s,s)}{\partial t}=1\quad (s\in I).
$$
\end{definition}

We note that Cauchy's function always exists and is unique.

\begin{definition}
A function $G:[a,b]^2\to \R$ is called \textit{Green's function} of boundary value problem
\begin{equation}\label{eq4}
(Lx)(t)=f(t)\quad (t\in I),\quad x(a)=0,\;x(b)=0\quad (a,b\in I),
\end{equation}
provided that it satisfies the following conditions: 1) $G$ is continuous on $[a,b]^2;$

2) $\frac{\partial G(\cdot,s)}{\partial t}$ is absolutely continuous in the triangles $a\leqslant s<t\leqslant b$ and $a\leqslant t<s\leqslant b,$ and
$$\frac{\partial G(s+,s)}{\partial t}-\frac{\partial G(s-,s)}{\partial t}=1;$$

3)$(LG)(\cdot ,s)=0\;\text{при}\;t\ne s;$

4) $G(a,s)=0,\;G(b,s)=0.$
\end{definition}

If the boundary value problem (\ref{eq4}) has the unique solution $x$, then it has the unique Green's function, and $x$ admits presentation
$$x(t)=\displaystyle{\int_a^b} G(t,s)f(s)\,ds.$$ Also, one has the following identity
$$
G(t,s)= 
\begin {cases} 
-\frac {C(b,t)C(s,a)}{C(b,a)},&\text{если}\; a \leqslant s < t ,\\ 
-\frac {C(t,a)C(b,s)}{C(b,a)},&\text{если}\; t\leqslant s \leqslant b,
\end {cases}
$$
which implies that if $C(t,s)>0,$ for $a\leqslant s<t\leqslant b,$ then $G(t,s)<0$ for $(t,s)\in (a,b)^2.$


\subsection{Disconjugacy and Sturm theorems}


We will need the following results due to Sturm: Separation of zeros theorem and Comparison theorem (see, e.g., \cite[p. 252]{stepvv},\,\cite[p. 81]{hart70}).

\begin{theorem}[Separation of zeros]
\label{thsep} 
Let $a$, $b\in I$, suppose that $x$ is a solution of equation {\rm (\ref{eq1})} such that $x(a)=x(b)=0,$   $x(t)\ne 0$ for any $t\in (a,b).$ Then any other solution of (\ref{eq1}), linearly independent with $x$, has only one zero in  $(a,b).$
\end{theorem}

\begin{proof}
Suppose that $y$ is a solution of equation (\ref{eq1}) linearly independent with $x$ and such that $y(t)\ne 0$ on $(a,b).$ Since $y(a)\ne 0,\;y(b)\ne 0$ (due to linear independence of $x$ and $y$),
$y(t)\ne 0$ on $[a,b].$ Define $h(t):= -W(t)/y(t),$ where Wronskian $W$ of $\{x,y\}$ is continuous and nowhere zero on $[a,b]$. Then $h(t)\ne 0$ on $[a,b].$  
Without loss of generality $h(t)>0$ on $[a,b].$ Since $h=(x/y)',$ we have
$\int_a^bh(t)\,dt>0$ and, at the same time, $$\int_a^bh(t)\,dt=\dfrac{x(b)}{y(b)}-\dfrac{x(a)}{y(a)}=0.$$ The latter implies that $y(t^*)=0$ at some $t^*\in (a,b).$
If $y(t_*)=0$ at some $t_*\ne t^*,$ then, as we have already proved, $x$ would have a zero in $(a,b),$ which contradicts to our assumptions.
\end{proof}

Let $a\in I,$ $x$ be a solution of equation (\ref{eq1}) such that $x(a)=0$. A point $\rho_+(a)>a$ $\bigl(\rho_-(a)<a\bigr)$ is called {\it right} ({\it left}) 
{\it conjugate point of $a$}  if
$$
x(\rho_{\pm}(a))=0,\; x(t)\ne 0\quad \text {in}\quad (a,\rho_+(a))\quad  \bigl(\text{in}\;(\rho_-(a),a)\bigr).
$$
If $x(t)\ne 0$ on $(a,\beta)$ $\bigl(\text{respectively}, (\alpha , a)\bigr),$ we define $\rho_+(a)=\beta\quad \bigl(\rho_-(a)=\alpha\bigr)$

\begin{corollary}\label{incrrho}
{\it Suppose that $\rho_+(t)\ne\beta,\;\rho_-(t)\ne\alpha$ for all $t\in\ I.$ Then functions  $\rho_{\pm}$\, are strictly increasing. 
Furthermore, $\rho_+\bigl(\rho_-(t)\bigl)=\rho_-\bigl(\rho_+(t)\bigl)=t\quad (t\in I),$ i.e., the functions $\rho_{\pm}$ are the inverses of each other and map continuously any interval in $I$ to an interval in $I.$}
\end{corollary} 


\begin{proof}
Let $t_2>t_1,\;x(t_1)=y(t_2)=0$\; ($x$ and $y$\, are solutions of {\rm (\ref{eq1})}). Suppose that $\rho_+(t_2)\leqslant \rho_+(t_1).$  The equality here, meaning that
$x(\rho_+(t_2))=y(\rho_+(t_1))=0$, contradicts to the definition of a conjugate point. Meanwhile, the strict inequality contradicts to Theorem \ref{thsep} (since $y$ would have two zeros between two consecutive zeros of $x.$) Consequently, $\rho_+(t_2)> \rho_+(t_1).$  

The proof for function $ \rho_-$ is similar. The proof of the second statement follows from the definition of conjugate points and properties of strictly monotone functions.
\end{proof}

\begin{remark} Note that if $\rho_+(a)=\beta$ or $\rho_-(a)=\alpha$ for a certain $a,$  then functions $\rho_{\pm}$ might not be monotone on interval $I,$ but only on $(\alpha,b)$ (or on $(a,\beta)$),
where $$b=\inf\{t:\rho_+(t)=\beta\}\;\Bigl(a=\sup\{t:\rho_-(t)=\alpha\}\Bigr).$$
For instance, equation
\begin{equation}
\label{eq3a}
x''-\frac{A\,\sh\,t}{A\,\ch\,t-1}\,x'+\frac{1}{A\,\ch\,t-1}\,x=0\quad (A\geqslant 2, \;I=(-\infty,+\infty))
\end{equation}
has a solution
$$
x(t)=\frac{A-\ch\,a}{A\ch\,a-1}\,\sh\,t+\frac{\sh\,a}{A\ch\,a-1}(\ch\,t-A)
$$
which satisfies $x(a)=0,\;x'(a)=1.$ Therefore, we obtain
$$
\rho _{+}(t)=\left\{
\begin{array}{lcr}
\ln\,\frac{A-e^t}{1-Ae^t}\quad \text{if}\quad -\infty <t<\ln\,\frac{1}{A}, \\
+\infty \quad \text{if}\quad t\geqslant \ln\,\frac{1}{A}.
\end{array}
\right.
$$ 
Analogously, 
$$
\rho _{-}(t)=\left\{
\begin{array}{lcr}
\ln\,\frac{A-e^t}{1-Ae^t}\quad \text{if}\quad \ln\,A<b<+\infty, \\
-\infty \quad \text{if}\quad t\leqslant \ln\,A.
\end{array}
\right.
$$
The same situation holds for equation
\begin{equation}\label{eq3b}
x''-\frac{2(2t-b)}{t^2+(t-b)^2}\,x'+\frac{4}{t^2+(t-b)^2}\,x=0\qquad (b>0);
\end{equation}
here
$
\rho _{+}(t)=\left\{
\begin{array}{lcr}
\frac{b(t-b)}{2t-b}\quad \text{if}\quad t<\frac{1}{2}b, \\ 
+\infty \quad \text{if}\quad t\geqslant \frac{1}{2}b;
\end{array}
\right.
$
$
\rho _{-}(t)=\left\{
\begin{array}{lcr}
\frac{b(t-b)}{2t-b}\quad \text{if}\quad t>\frac{1}{2}b, \\ 
-\infty \quad \text{if}\quad \leqslant \frac{1}{2}b.
\end{array}
\right.
$
\end{remark}

\begin{definition}
We say that a differential equation (\ref{eq1}) is {\it disconjugate} on an open interval $J\subset I$ if any of its non-trivial solutions has at most one zero in $J$.
If the latter is the case, we say that $J$ is an interval of disconjugacy of equation  (\ref{eq1}).
\end{definition}

Thus, $J$ is an interval of disconjugacy of equation
 (\ref{eq1}) if and only if $\rho_{\pm}(a)\notin J$ for any $a\in J.$ 

It follows from the above representation of functions $\rho_{\pm}$ that the intervals for disconjugacy of equation (\ref{eq3a}) are
$$
\left[a,\ln\dfrac{A-e^a}{1-Ae^a}\right),\;\text{if}\;a<\ln\frac{1}{A};\quad [a,+\infty),\;\text{if}\;a\geqslant \ln\frac{1}{A},
$$
while for equation the intervals for disconjugacy of equation (\ref{eq3b}) are
$$
\left[a,\dfrac{b(a-b)}{2a-b}\right),\;\text{if}\;a<\frac{1}{2}b;\quad [a,+\infty),\;\text{if}\;a\geqslant \frac{1}{2}b.
$$

\begin{definition}
We denote by $\gT(J)$ the class of linear differential operators $L$ such that the corresponding homogeneous equation $Lx=0$ is disconjugate on interval $J\subset I.$  
\end{definition}

Let $(a,b)\subset I,$ suppose that $a_n\to a+,\;b_n\to b-$ ($a_n\to -\infty,\;b_n\to +\infty$ in the case $a=\alpha=-\infty,\;b=\beta =+\infty$). Then
\begin{equation}\label{eqdop}
\gT\bigl((a,b)\bigr)=\bigcap\limits_{n=1}^{\infty}\gT\left(\bigl[a_n,\,b_n\bigr]\right)=\bigcap\limits_{n=1}^{\infty}\gT\left(\bigl(a_n,\,b_n\bigr)\right).
\end{equation}

As follows from the definition of the property of disconjugacy and the definition of Cauchy's function, if 
equation (\ref{eq1}) is disconjugate on interval $J=[a,b)\subset I,$ then $C(t,s)>0$ in the triangle
$a\leqslant s<t<b.$ The disconjugacy of equation (\ref{eq1}) on an interval $[a,b]$ implies the existence of the unique solution of problem (\ref{eq4}), so the Green's function of this problem satisfies $G(t,s)<0$ on $(a,b)^2.$

\begin{theorem}[Comparison theorem]
\label{thcomp} 
Let
$$
L_i\,y:= y''+p(t)\,y'+q_i(t)\,y=0,\quad i=1,2\quad\text{и}\quad q_1(t)\leqslant  q_2(t)\quad (t\in I).
$$ 
If $L_2\in \gT(J),$ then
$L_1\in \gT(J)$.
\end{theorem}

\begin{proof}
Let $L_2\in \gT(J).$ Suppose that $L_1\notin \gT(J).$  Then there exist a solution $x$ of equation $L_1x=0$, and points $a,b\in J,\;a<b,\;x(a)=x(b)=0,\;x(t)>0\;(t\in (a,b)).$
Let us define the solution $y$ of equation $L_2y=0$ by initial values $y(a)=0,\;y'(a)=x'(a)\,(>0).$ Since $L_2\in \gT(J),$ then $y(b)>0.$ Let $h:= y-x.$ 
Then $0=L_2y-L_1x=L_2h-(q_1-q_2)x, \;h(a)=0,h'(a)=0.$ Also, note that $h(b)>0.$ Thus, function $h$ is a solution to problem
$$
(L_2h)(t)=(q_1(t)-q_2(t))x(t),\quad h(a)=h'(a)=0,
$$
i.e.,
$$
h(t)=\int_a^tC_2(t,s)\bigl(q_1(s)-q_2(s)\bigr)x(s)\,ds\leqslant 0, 
$$ 
since Cauchy's functions $C_2(t,s)$ of equation $L_2y=0$ is positive for all $a\leqslant s<t\leqslant b,\; x(s)>0,$\\ $\bigl(q_1(s)-q_2(s)\bigr)\leqslant 0$ for all $a<s \leqslant t\leqslant b.$
The inequality $h(b)\leqslant 0$ contradicts to the inequality $h(b)>0$ obtained previously, so $L_1\in \gT(J).$
\end{proof}

\subsection{Properties of class $\gT(J)$}

\begin{theorem}\label{th21} 
Let $a$, $b \in I.$ Then 
$
\gT\bigl((a,b)\bigr) = \gT\bigl((a,b]\bigr) = \gT\bigl([a, b)\bigr).
$
\end{theorem}

\begin{proof} 
It suffices to check inclusion $\gT\bigl((a,b)\bigr) \subset \gT\bigl([a,b)\bigr),$ the rest follows by symmetry.

Let
$L\in \gT\bigl((a,b)\bigr).$ Suppose that $L\notin \gT\bigl([a,b)\bigr).$ Then there exist a solution $v$ of equation (\ref{eq1}) such that $v(a)=v(c)=0\; (a<c<b)$, $v(t)>0$ in $(a,c)$ (note that $v$ 
has at least two zeros in $[a,b),$ and at most one zero in $(a,b)$). By definition, $c=\rho_+(a)\quad \bigl((a=\rho_-(c)\bigr).$ Let us choose  $c_1 \in (c,b)$ so that $v(t)<0$ in $(c,c_1)$. We put $a_1=\rho _{-}(c_1).$ According to Corollary \ref{incrrho} $a<a_1<c.$ 
Let $x$ be the corresponding solution of equation (\ref{eq1}), i.e., $x(c_1)=x(a_1)=0,\quad x(t)>0$ in $(a_1,c_1).$
Since $L\in \gT\bigl([a_1,c_1]\bigr),$ there exists the unique solution $y$ of equation (\ref{eq1}) that satisfies $y(a_1)=y(c_1)=1.$ We have $y(t)>0$ on $[a_1,c_1]$  
(as a continuous function taking the same values at the endpoints of the interval, function $y$ can have only even number of zeros, hence, due to disconjugacy, none of them). We note that $y$ is linearly independent with $v$ and with $x.$ According to Theorem \ref{thsep} $y$ has exactly one zero in both intervals $(a,a_1)$ and $(c,c_1),$ that is, $y$ has two zeros in $(a,b).$ The latter contradicts to disconjugacy of equation (\ref{eq1}) on $(a,b)$.
\end{proof}

\begin{theorem}\label{thnefkrit}
$1.$ If there exists a solution of equation (\ref{eq1}) that is nowhere zero on $[a,b]\subset  I \\ \bigl((a,b)\subset I\bigr),$ 
then $L\in \gT\bigl([a,b]\bigr)\quad \bigl(L\in\gT\bigl((a,b)\bigr)\bigr).$

$2.$ If $L\in \gT([a,b],\;[a,b]\subset  I)\quad \bigl(L\in \gT([a,b)),\;[a,b)\subset I \bigr),$ then there exists a solution of equation (\ref{eq1}) that is nowhere zero on $[a,b]\quad 
\bigl((a,b)\bigr).$
\end{theorem}

\begin{proof}
1. The statement follows immediately from Theorem \ref{thsep}.

2. Let $J=[a,b],\;[a,b]\subset  I.$ Let us determine solutions $y_1(t)$ and $y_2(t)$ by initial conditions $y_1(a)=0,\, y_1'(a)=1$ and
$y_2(b)=$~$0,$ $y_2'(b)=-1.$ 
Since $L\in\gT([a, b]),$ one has 
$y_1(t)>0$ if $t\in (a, b]$, and $y_2(t)>0$ if $t\in [a, b)$.
The solution $y_1(t) + y_2(t)$ is the one required. If $L\in \gT([a,b)),$ then the required solution is $y_1.$ 
\end{proof}

It is possible that there are no solutions preserving sign on $[a,b)$. For instance, consider
$${L:=\frac{d^2}{dt^2}+1}\in\gT([0, \pi)).$$ Then any solution of equation $Lx=0$ has precisely one zero in $[0, \pi)$.

\section{Applications of disconjugacy}

Below we prove two theorems which demonstrate the role of the property of disconjugacy in the theory of differential equation (\ref{eq1}). These are the Factorization theorem (on representation of a linear ordinary differential operator $L$ as the product of linear differential operators of the first order, see, e.g., \cite{pol24},\cite{mam31}) and the generalized Rolle's Theorem (see, e.g., \cite[p. 63]{ps782}).

\begin{theorem}[Factorization theorem]\label{thPM} 
Suppose $J=[a,b]\subset I$ or $J=(a,b)\subset I.$ One has $L\in \gT(J),$ if and only if there exist functions $h_i, i=0,1,2$ such that $h_0',h_1$ are absolutely continuous,
$h_2$ is summable on $J,\;h_i(t)>0,\;h_0(t)h_1(t)h_2(t)\equiv 1$ on $J$, and
\begin{equation}
\label{PMeq1}
(Lx)(t)=h_2(t)\frac{d}{dt}h_1(t)\frac{d}{dt}h_0(t)x(t)\quad (t\in J,\quad x'\text{ is absolutely continuous on } J).
\end{equation}
\end{theorem}

\begin{proof} 
\textit{Necessity.} Let $L\in \gT(J).$ According 
to Theorem \ref{thnefkrit} there exists a 
solution $y$ of equation (\ref{eq1}) such that $y(t)>0$ on $J.$ Let $u$ be a solution of equation (\ref{eq1})
linearly independent with $y$ and such that $w(t):= [y,u](t)>0.$ Let us consider the following linear differential operator of the second order
$$
\widehat Lx:= \frac{w}{y}\frac{dt}{dt}\frac{y^2}{w}\frac{dt}{dt}\frac{x}{y}.
$$
Since functions $y,u$ form a fundamental system of solutions of both equation (\ref{eq1}) and equation $\widehat Lx=0,$  the top coefficient in $\widehat L$ is
equal to $\frac{w}{y}\frac{y^2}{w}\frac{1}{y}\equiv 1,$ then $Lx\equiv \widehat Lx.$ These conditions are satisfied if $h_0=\frac{1}{y},\;h_1=\frac{y^2}{w},\;h_2=\frac{w}{y}.$

\textit{Sufficiency.} Suppose that we have identity (\ref{PMeq1}). Then function $y(t):=\frac{1}{h_0(t)}>0\;(t\in J)$ is a solution of equation (\ref{eq1}) satisfying the conditions of Theorem \ref{thnefkrit}. This implies that $L\in \gT(J).$
\end{proof}

\begin{theorem}[Generalized Rolle's theorem]
\label{thRol} 
Let $J=[a,b]\subset I$ or $J=(a,b)\subset I,\quad L\in \gT(J).$ Suppose that function $u$ has absolutely continuous on $J$ first derivative, and function $Lu$ is continuous. If there exist $m$ ($m\geqslant 2$) geometrically distinct zeros of $u$ in $J$, then $Lu$ has at least $m-2$ geometrically distinct zeros in $J$.
\end{theorem}

\begin{proof} 
According to Theorem \ref{thPM} one has representation (\ref{PMeq1}). Function $h_0u$ has $m$ geometrically distinct zeros in $J$. By Rolle's theorem both functions $\frac{d}{dt}h_0u$ and $h_1\frac{d}{dt}h_0u$ have at least $m-1$ geometrically distinct zeros in $J$. Again, by Rolle's theorem function $Lu$ has at least $m-2$ geometrically distinct zeros in $J$. The proof is complete.
\end{proof}

\section{Criteria for disconjugacy}

\subsection{Basic criteria}
Below we formulate some known criteria for disconjugacy based on Theorems \ref{thcomp} and \ref{thnefkrit}.

\begin{criteria}\label{const}
Let $I=(-\infty,\,+\infty)$, $p(t)\equiv p=\const$, $q(t)\equiv q=\const$. Then differential equation $(\ref{eq1})$ having constant coefficients $p(t) \equiv p$, $q(t) \equiv q$ is disconjugate on $I$ if and only if the roots of its characteristic equation
 $\lambda ^2+p\lambda +q=0$ are real.
\end{criteria}

\begin{proof}
Let $\nu$ be a real root of the characteristic equation. Then function $x(t):= e^{\nu t}$ is a solution of equation (\ref{eq1})
nowhere vanishing on $I.$ 
According to the first statement of Theorem \ref{thnefkrit}, equation (\ref{eq1}) is disconjugate on $I.$

Conversely, let (\ref{eq1}) be disconjugate on $I.$ Suppose that the characteristic equation has roots $\gamma \pm\delta i, \delta\ne 0.$ Then solution
$x(t)=e^{\gamma t}\cos\,\delta t$ of equation (\ref{eq1}) has infinitely many zeros in $I$, which contradicts to its disconjugacy on $I.$
\end{proof}

Let us consider equation
\begin{equation}\label{euler1}
x''+\frac{p}{t}x'+q(t)x=0\quad \bigl(t\in I:=(0,+\infty)\bigr),\quad\text{where}\;p=\const .
\end{equation}

\begin{criteria}\label{euler}
If $q(t)\leqslant \frac{(p-1)^2}{4t^2},$ then equation $(\ref{euler1})$ is disconjugate on $I:= (0,+\infty).$
\end{criteria}

\begin{proof}
Euler equation $x''+\frac{p}{t}x'+\frac{(p-1)^2}{4t^2}x=0$ is disconjugate on $I$ by Theorem \ref{thnefkrit} since it has solution $x(t)=t^{\frac{1-p}{2}},$ 
which is nowhere equal to zero on $I$ (let us also take into account (\ref{eqdop})). 
According to Theorem \ref{thcomp} equation (\ref{euler1}) is also disconjugate on this interval.
\end{proof}

The next sufficient condition of disconjugacy is due to A.M.~Lyapunov
\cite{cop71}.

\begin{criteria}\label{ljap}
Let $p(t)\equiv 0, \quad q(t) \geqslant 0$ and $\int_a^b q(t)dt \leqslant \frac {4}{b-a}.$ Then $L \in \gT([a, b]).$
\end{criteria}

\begin{proof}
Suppose that equation (\ref{eq1}) possesses a non-trivial solution $y(t)$ having two zeros in $[a,b].$
Since $y$ can not have multiple roots, we may assume, without loss of generality, that
\begin{equation}\label{ljap1}
y(a) = y(b) = 0. 
\end{equation}
Function $y$, as a solution of boundary value problem (\ref{eq1}), (\ref{ljap1}), satisfies the following integral equation 
\begin{equation}\label{ljap2}
y(t)= -\int_a^b G(t,s)q(s)y(s) ds, 
\end{equation}
where
$$
G(t,s)= 
\begin {cases} 
-\dfrac {(b-t)(s-a)}{b-a},&\text{if}\; a \leqslant s < t ,\\ 
-\dfrac {(t-a)(b-s)}{b-a},&\text{if}\; t\leqslant s \leqslant b 
\end {cases}
$$
is Green's function of equation $y''=0$ with boundary conditions (\ref{ljap1}). It is immediate that for $t \ne s$ 
\begin{equation}\label{ljap3}
|G(t, s)| < \frac {(b-s)(s-a)}{b-a}.
\end{equation}
Let $\max\limits_{s \in [a, b]}|y(s)| = |y(t^*)|$. Then (\ref{ljap2}) and (\ref{ljap3}) 
\begin {multline*}
|y(t^*)|= \left | \int_a^b G(t^*, s)q(s)y(s)ds \right | \leqslant \int_a^b \left |G(t^*, s) \right ||y(s)|q(s)ds < \\ 
|y(t^*)| \left | \int_a^b \frac {(b-s)(s-a)q(s)}{b-a}ds \right | \leqslant \frac {b-a}{4} \int_a^b q(s)ds 
\end {multline*} 
since $(b-s)(s-a)\! \leqslant \!\dfrac {(b-a)^2}{4}$ for $s\! \in \![a, b].$ Therefore, $1\!\! < \!\!\dfrac {b-a}{4} \displaystyle{\int_a^b} q(s)ds,$ which contradicts to the conditions of the theorem.
\end{proof}

\begin{corollary}\label{slthljap}
If $p(t)\equiv 0,$ and
$
\int_a^b q_{+}(t)\,dt\leqslant \frac{4}{b-a},
$
then $L\in \gT([a,b])$. 

Here $q_{+}(t):=q(t)$ if $q(t)>0$,  $q_{+}(t):=0$ if $q(t)\leqslant 0.$
\end{corollary}

\begin{proof}
As we have already proved, $L_{+}:= \frac{d^2}{dt^2}+q_{+}\in \gT([a,b])$. At the same time, since $q(t)\leqslant q_{+}(t)$, one has $L\in \gT([a,b]).$
\end{proof}

\begin{remark}\label{z1} 
We note that the constant $4$ is the formulation of Criterion $\ref{ljap}$ is sharp. 
\end{remark}

The latter follows from the next example.
Suppose that function $v$ is twice continuously differentiable on $[0,1]$ and
$$
v(t)=t \quad (0 \leqslant t \leqslant \frac{1}{2}-\delta),\quad
v(t)=1-t \quad \text{ if } \quad t>\frac{1}{2}+ \delta, 
$$
$$
v(t)>0, \quad v''(t)<0 \quad \text{ if } \quad \frac{1}{2}-\delta < t < \frac{1}{2}\!+\! \delta.
$$
Define
$$
q(t)= 
\begin {cases} 
-\frac {v''(t)}{v(t)},&\text{if} \; t \in (0, 1) ,\\ 
0,&\text{if} \;t = 0,  \; t = 1. 
\end {cases}
$$
Clearly, $q$ is continuous, $q(t)\! \geqslant \!0$ on $[0, 1];\; L\! := \!\frac {d^2}{dt^2} + q(t) \notin \gT([0, 1]),$ since equation 
$Ly=0$ has solution $y=v(t)$ which has two zeros in $[0,1].$ However,  
$$
\frac {v''}{v} = \left ( \frac {v'}{v} \right )' + {\left (\frac {v'}{v} \right )}^2 \geqslant \left ( \frac {v'}{v} \right )',
$$
so the value of integral
$$
\int_0^1 q(t)dt = -\int_{\frac{1}{2}-\delta}^{\frac{1}{2}+\delta} \left ( \frac {v'}{v} \right )'dt= 
\left. -\frac {v'}{v}  \right |_{\frac{1}{2}-\delta}^{\frac{1}{2}+\delta} = \frac{4}{1-2\delta}
$$ 
can be made arbitrarily close to $4$ by choosing sufficiently small $\delta.$

\subsection{Semi-effective criteria}

Criterion \ref{thnefkrit} is an example of a \textit{non-effective} criterion of disconjugacy, i.e., a criterion formulated in terms of solutions of equation
(\ref{eq1}) rather than in terms of the coefficients of this equation. 

Let us now formulate a necessary and sufficient condition of disconjugacy of equation
(\ref{eq1}) belonging to Valle-Poussin \cite{vp29}; this criterion may be called \textit{semi-effective} \cite{lev69}, i.e., it is effective as a necessary condition, but non-effective as a sufficient condition.
Although this criterion is not expressed in terms of the coefficients of equation (\ref{eq1}), it can be used to obtain sufficient conditions of disconjugacy formulated in terms of the coefficients. 

\begin{theorem}\label{thVP}   
Let $[a, b] \subset I.$ One has $L \in \gT([a, b])$
if and only if there exists function
$v$ having first derivative absolutely continuous on $[a, b]$ and such that
\begin{equation}\label{vp1}
v(t) > 0\quad  (a < t \leqslant b), \qquad  Lv \leqslant 0 \quad \text{a.e. on}\; [a,b]. 
\end{equation} 
\end{theorem}
\begin{proof}
Necessity follows from Theorem \ref{thnefkrit}. Let us show that the conditions of the theorem are sufficient. In the case $v(a) = 0$ let us put
$\widetilde v(t) = v(t)+\varepsilon u(t)$, where $\varepsilon > 0,$ and $u(t)$ is the solution of equation (\ref{eq1}) with initial conditions
$u(a)=1, \; u'(a)=0$. For $\varepsilon$ sufficiently small we have $\widetilde v(t) > 0$ on $[a, b].$
Hence we may assume, without loss of generality, that
$v(t) >0$ on $[a, b].$  Let us consider equation
\begin{equation}\label{vp2}
Mx:= x''+px'-\frac{v''+pv'}{v}x=0.
\end{equation} 
According to Theorem \ref{thnefkrit} $M \in \gT([a, b])$ (since equation (\ref{vp2}) has solution $v$ positive on $[a, b]$). 
By our assumptions $v''(t) + p(t)v'(t)+q(t)v(t) \leqslant 0,$ i.e., $- \frac {v''(t)+p(t)v'(t)}{v(t)} \geqslant q(t),$  a.e. on $[a,b] .$  
The statement of the theorem now follows from Theorem \ref{thcomp}.
\end{proof}

The proof of the next statement follows the same argument.

\begin{theorem}\label{VP'} 
If there exists function $v$ having first derivative absolutely continuous on $[a, b)$ ans such that 
\begin{equation}\label{vp3}
v(t) > 0\quad (a < t < b ), \qquad Lv \leqslant 0 \quad \text{ a.e. on}\quad (a,b), 
\end{equation} 
then  $L \in \gT([a, b)).$
\end{theorem}

\subsection{Effective criteria}

By choosing a particular `test' function $v$, one can obtain an effective criterion for disconjugacy.

\begin{criteria}\label{A}
If $q(t)\leqslant 0$ on $[a,b]\subset I$ $\bigl(\text{or on }(a,b)\subset I\bigr)$, then $L\in\gT([a,b])$ $\left(\text{resp., }L\in \gT\bigl((a,b)\bigr)\right).$
\end{criteria}

\begin{proof}
We put $v(t)\equiv 1$ and then use Theorem \ref{thVP} \;(Theorem \ref{VP'}).
\end{proof}

\begin{criteria}\label{B}
Suppose that $p(t)=O(t-a)$ if $t\to a+,\;p(t)=O(b-t)$ if $t\to b-$ $($in particular, $p(t)\equiv 0).$ If 
\begin{equation*}
\frac{\pi}{b-a}\cot\,\frac{\pi (t-a)}{b-a}p(t)+q(t)\leqslant \frac{\pi ^2}{(b-a)^2},
\end{equation*}
then $L \in \gT\bigl([a, b)\bigr).$
\end{criteria}

\begin{proof}
Let us choose $v(t)\equiv \sin\,\frac{\pi (t-a)}{b-a} $ and then use Theorems  \ref{VP'} and \ref{th21}.
\end{proof}

\begin{criteria}\label{C}
Suppose that we have inequality
\begin{equation}\label{C1}
|p(t)|\cdot \left | \frac{b+a}{2}-t\right |+|q(t)|\cdot \frac{(b-t)(t-a)}{2}\leqslant 1
\end{equation}
or inequality
\begin{equation}\label{C2}
\frac{b-a}{2}\underset{t\in (a,b)}{\essup}\,|p(t)|+\frac{(b-a)^2}{8}\underset{t\in (a,b)}{\essup}\,|q(t)|\leqslant1.
\end{equation}
Then $L \in \gT\bigl([a, b)\bigr).$
\end{criteria}

\begin{proof}
Indeed, we take $v(t)\equiv \frac{(b-t)(t-a)}{2}$ and then refer to Theorems \ref{th21} and \ref{VP'}.
\end{proof}

Let us note that inequality (\ref {C2}) implies inequality (\ref{C1}).
Let $P(t,\lambda):= \lambda ^2+p(t)\lambda +q(t)$ be the `characteristic' polynomial.

\begin{criteria}\label{D}
If there exists $\nu \in \mathbb R$ such that $P(t,\nu)\leqslant 0\quad (t\in (-\infty,+\infty)),$ then equation (\ref{eq1}) is disconjugate on $(-\infty,+\infty).$
\end{criteria}

\begin{proof}
One has $v(t):= e^{\nu t}>0$ and $(Lv)(t)=e^{\nu t}P(t,\nu)\leqslant 0$ on  $(-\infty,+\infty).$ 
The rest follows from Theorem \ref{thVP}.
\end{proof}

Let us now formulate criteria that can be obtained from Theorem
 \ref{thVP}\;(Theorem  \ref{VP'}) using a `test' function depending on coefficients of equation 
(\ref{eq1}).

{\bf{$1^o.$}}
Let us consider equation
\begin{equation}\label{ord2.50}
\widetilde Lx:=  x''+Px'+Qx=0 
\end{equation}
having constant coefficients $P$ and $Q$, in assumption that it is disconjugate on $[a,b).$
Let $v$ be the solution of boundary value problem  $\widetilde Lv=-1,\quad v(a)=v(b)=0,$ let $\widetilde C(t,s)$
be Cauchy's function of equation (\ref{ord2.50}). Then
$$
\widetilde C(t,s)>0\quad (a\leqslant s<t<b)\quad \text{and}\quad v(t)=\int_a^b M(t,s)\,ds>0\quad (t\in (a,b)),
$$
where
$$
M(t,s):=\left\{
\begin{array}{rcl}
\dfrac{\widetilde C(b,t)\cdot \widetilde C(s,a)}{\widetilde C(b,a)},\quad a\leqslant s\leqslant t \leqslant b,\\
\dfrac{\widetilde C(t,a)\cdot \widetilde C(b,s)}{\widetilde C(b,a)},\quad a\leqslant t< s \leqslant b
\end{array}
>0,\quad (t,s)\in ((a,b)\times (a,b)).
\right.
$$
Since $(Lv)(t)=-1+\bigl(p(t)-P\bigr)v'(t)+\bigl(q(t)-Q\bigr)v(t),$ inequality $(Lv)(t)\leqslant 0$
is satisfied if
\begin{equation}\label{gener}
\bigl(p(t)-P\bigr)\int_a^b \frac{\partial M(t,s)}{\partial t}\,ds+\bigl(q(t)-Q\bigr)\int_a^b M(t,s)\,ds)ds\leqslant 1, \quad t\in (a,b).
\end{equation}
As a result, we get the following statement.

\begin{criteria}\label{XA1}
If $(\ref{gener}),$ then (\ref{eq1}) is disconjugate on $[a,b).$
\end{criteria}

The special choice of coefficients $P$ and $Q$ 
can lead to criteria for disconjugacy that are more subtle than the ones formulated above.

{\bf{$2^o.$}}
Consider the particular case $Q=0.$ We have
$$
M(t,s)=\left\{
\begin{array}{rcl}
\dfrac{(1-e^{-P(b-t)})(1-e^{-P(s-a)})}{P(1-e^{-P(b-a)})}\quad (s\leqslant t), \\
\dfrac{(1-e^{-P(t-a)})(1-e^{-P(b-s)})}{P(1-e^{-P(b-a)})}\quad (s>t).
\end{array}
\right.
$$
It is immediate that
\begin{multline*}
 v(t)=
\frac{(1-e^{-P(b-t)})(t-a-\frac{1}{P}(1-e^{-P(t-a)}))}{P(1-e^{-P(b-a)})}+\\ 
\frac{(1-e^{-P(t-a)})(b-t-\frac{1}{P}(1-e^{-P(b-t)}))}{P(1-e^{-P(b-a)})}\leqslant 
\frac{2(\frac{b-a}{2}-\frac{1}{P}(1-e^{-P\frac{b-a}{2}}))}{P(1+e^{-P\frac{b-a}{2}}))}, \\
v'(t)=\frac{P\left((b-t)e^{-P(t-a)}-(t-a)e^{-P(b-t)}\right)+e^{-P(b-t)}-e^{-P(t-a)}}{P\left(1-e^{-P(b-a)}\right)}, \\
|v'(t)|\leqslant \frac{|P(b-a)+e^{-P(b-a)}-1|}{P(1-e^{-P(b-a)})}.
\end{multline*}
Since condition $Lv\leqslant 0$ is now equivalent to inequality 
$\bigl(p(t)-P\bigr)v'(t)+q(t)v(t)\leqslant 1,$ 
we get the following criterion.

\begin{criteria}\label{XA2}
If
$$
|p(t)\!-\!P| \frac{|P(b-a)\!+\!e^{-P(b-a)}-1|}{P(1-e^{-P(b-a)})}\!+\!|q(t)| 
\frac{2(\frac{b-a}{2}-\frac{1}{P}(1-e^{-P\frac{b-a}{2}}))}{P(1+e^{-P\frac{b-a}{2}}))}\!\leqslant \!1
$$ 
$(a<t<b),$ then equation (\ref{eq1}) is disconjugate on $[a,b).$
\end{criteria}

{\bf{$3^o.$}}
If we take, instead of auxiliary equation (\ref{ord2.50}), equation $\widetilde Lx:=  x''+p(t)x'=0$,
and take as $v$ the solution of problem
$\widetilde Lv=-1,\quad v(a)=v(b)=0$,
we obtain the following criterion.

\begin{criteria}\label{XA3}
If $
q(t)\,\int_a^b M(t,s)\,ds\leqslant 1, \quad t\in (a,b),
$
where
$$
M(t,s)=\left\{
\begin{array}{rcl}
\dfrac{\int_t^be^{-\int_t^{\sigma}p(\mu)\,d\mu}d\sigma\cdot \int_a^se^{-\int_a^{\sigma}p(\mu)\,d\mu}d\sigma} 
{\int_a^be^{-\int_a^{\sigma}p(\mu)\,d\mu}d\sigma}\quad (s\leqslant t), \\
\dfrac{\int_a^te^{-\int_a^{\sigma}p(\mu)\,d\mu}d\sigma\cdot \int_s^be^{-\int_s^{\sigma}p(\mu)\,d\mu}d\sigma} 
{\int_a^be^{-\int_a^{\sigma}p(\mu)\,d\mu}d\sigma}\quad (t<s),
\end{array}
\right.
$$
then equation (\ref{eq1}) is disconjugate on $[a,b).$
\end{criteria}

\section{A new criterion for disconjugacy}

\textbf{1.~}In what follows, we derive a second order criterion for disconjugacy on the whole real axis $\mathbb R$.
Let us consider differential equation
\begin{equation}\label{ord2.8}
\widetilde Lx:=  x''+px'+qx=0 
\end{equation}
having constant coefficients $p$ and $q$. As was shown before (see Criterion \ref{const}), disconjugacy of equation (\ref{ord2.8})
on $\mathbb R$ is equivalent to inequality $p^2-4q\geqslant 0.$

We associate to equation (\ref{ord2.8}) a point $\widetilde{\mathcal L}=(p,q)$ in the $(p,q)$-plane $\Pi$.
Let
$$
\mathfrak N:= \{(p,q):p^2-4q\geqslant 0\},\qquad \mathfrak O:= \mathbb R^2\setminus \mathfrak N, \qquad \mathfrak M_{\pm} (\gamma):=\{(p,q): q\leqslant -\gamma ^2\pm\gamma \,p\}.
$$
Then according to Criterion \ref{const} 
$$
\widetilde L\in\gT \bigl((-\infty,\,+\infty)\bigr) \Longleftrightarrow \widetilde{\mathcal L}\in\mathfrak N.
$$

Let us now consider differential equation
\begin{equation}\label{ord2.9}
Lx:=  x''+p(t)x'+q(t)x=0
\end{equation}
with locally integrable coefficients on $(-\infty,\,+\infty)$. 
Every equation of form (\ref{ord2.9}) gives rise to a `curve' $G_L=\{t: (p(t),q(t))\}$ in plane $\Pi$ (we use quotation marks since this curve is determined up to a set of measure zero),
more precisely, it determines a motion $DG_L$ along this curve. 

\vspace*{2mm}

Now, the inclusion $G_L \subset \mathfrak N$ is neither necessary nor sufficient for the disconjugacy of equation (\ref{ord2.9}) on $\mathbb R:$
for instance, equation $x''-\frac{t}{2}x'+\frac{t^2}{16}x=0,$ despite the above inclusion, possesses a solution $u(t)\equiv e^{\frac{t^2}{8}}\sin\,\frac{t}{2},$ 
which has two zeros on each interval $[2k\pi,\,2(k+1)\pi]\;(k\in \mathbb Z)$, i.e., this equation is not disconjugate.
At the same time, equation $$ x''+tx'+\left(\frac{t^2}{4}+\frac{1}{2}\right)x=0$$ having solution $u(t)\equiv e^{-\frac{t^2}{4}}>0 \; (t\in\mathbb R),$
is disconjugate on $\mathbb R$ by Theorem \ref{VP'}, however $G_L\subset \mathfrak O.$ The same is true for a more general equation
\begin{equation}\label{ord2.12}
x''+p(t)x'+\left(\frac{p^2(t)}{4}+\frac{1}{2}p'(t)\right)x=0,\qquad p'(t)>0,\; t\in \mathbb R;
\end{equation}
it has a solution
$x=e^{-\frac{1}{2}{\displaystyle\int_0^t} p(s)\,ds}>0,\;t \in \mathbb R.$ 
Another example confirming that $G_L \subset \mathfrak N$ is neither necessary nor sufficient for disconjugacy is given by equation
\begin{equation}\label{ord2.11}
x''+\dfrac{\sin\,t}{2+\sin\,t}x=0, 
\end{equation}
which is disconjugate on $\mathbb R$ by Theorem  \ref{VP'}, but $G_L\bigcap \mathfrak N\ne\varnothing,\;G_L\bigcap \mathfrak O\ne\varnothing.$

Nevertheless, we have the following result.

\begin{theorem}\label{osnov}
{\it Suppose that one of the following conditions is satisfied:

$1)$ We have $p(t)\equiv p=\const,\;G_L\subset\mathfrak N.$

$2)$ $G_L\subset\mathfrak N, \;G_L$ is a line or line segment.
 
$3)$ We have $G_L\subset \mathfrak M_{+}(\gamma)$ (respectively, $G_L\subset \mathfrak M_{-}(\gamma)$) for a certain $\gamma \geqslant 0$.

$4)$ The function $p$ is differentiable,  $p'(t)\geqslant 0$ $ \bigl(p'(t)\leqslant 0\bigr)$ on  $\mathbb R$, and $G_L\subset\mathfrak N.$

$5)$ The function $p$ is differentiable,  $p'(t)\geqslant 0$ $\bigl(p'(t)\leqslant 0\bigr)$ on $\mathbb R$, and
$$
q(t)\leqslant \dfrac{p^2(t)}{4}+\frac{1}{2}p'(t)\qquad \Bigl(q(t)\leqslant \dfrac{p^2(t)}{4}-\frac{1}{2}p'(t)\Bigr).
$$

$6)$ Suppose that function $r:\mathbb R\to \mathbb R$ is continuous, function $p$ is differentiable and one of the following conditions is satisfied: 
\begin{equation}\label{dop21}
p'(t)\geqslant 2r(t)\quad (p'(t)\leqslant -2r(t))\;(t\in \mathbb R)  
\end{equation}
or
\begin{equation}\label{dop22}
p^2(t)-4p'(t)+r(t)\leqslant 0\quad (p^2(t)+4p'(t)+r(t)\leqslant 0)\; (t\in \mathbb R)
\end{equation}
and also
$q(t)\leqslant \frac{p^2(t)}{4}+r(t)$ $(t\in \mathbb R).$ 

\vspace*{1mm}

Then equation $(\ref{ord2.9})$ is disconjugate on $\mathbb R.$}
\end{theorem}

\begin{proof}
\,1) Let us define $v(t):= e^{-\frac{p}{2}t}>0.$ Then  $(Lv)(t)=e^{-\frac{p}{2}t}\left(q(t)- \frac{1}{4}p^2\right)\leqslant 0\;(t\in\mathbb R).$ 
We now apply Theorem \ref{VP'} to complete the proof.

2)\, The equation of such a line is either
$q(t)\equiv q=\const \leqslant 0$ for any $p(t)$, or
$p=p(t)$, $q=-\gamma ^2+k\,p(t)$, where $\quad |k|\leqslant \gamma$ ($\gamma >0$)
(if $k=\pm \gamma$ then the line is tangent to the parabola $q=\frac{1}{4}p^2$). In the first case the disconjugacy of equation (\ref{ord2.9}) on $\mathbb R$ follows from Theorem \ref{thcomp}.
In the second case the function $v(t):= e^{-kt}>0\;(t\in\mathbb R)$
satisfies condition
$(Lv)(t)=e^{-kt}(k^2- \gamma ^2)\leqslant 0\;(t\in\mathbb R).$  
Theorem \ref{VP'} now concludes the proof.

3) \,We put $v(t):=e^{-\gamma\,t}$ (accordingly, $v(t):=e^{\gamma\,t}$) and use Theorems \ref{thcomp} and \ref{VP'}.

4)\, Let $p'(t)\geqslant 0$, $L_2x:= x''+p(t)x'+\frac{p^2(t)}{4}\,x.$ We put $v(t):=\exp\left(-\frac{1}{2}{\displaystyle\int_0^t} p(s)\,ds\right).$  Then $v(t) >0,$ and
$$(L_2v)(t)=-\frac{1}{2}p'(t)\,e^{-\frac{1}{2}{\displaystyle\int_0^t} p(s)\,ds}\leqslant 0,\;t\in \mathbb R.$$ It follows that equation $L_2x=0$ is disconjugate on real line.
The disconjugacy of equation (\ref{ord2.9}) now follows from Theorem \ref{thcomp}.

If $p'(t)\leqslant 0$ then we put $y(t)=x(-t)$ and obtain equation $y''-p(t)y'+q(t)y=0$ of the form considered above. 

5) Suppose that $p'(t)\!\geqslant \!0.$ Then equation (\ref{ord2.12})
has solution
$$v(t)\!=\!\exp\left(-\frac{1}{2}{\displaystyle\int_0^t} p(s)\,ds\right)\!>~0,$$ hence it is disconjugate on $\mathbb R$.
Again, the disconjugacy of equation (\ref{ord2.9}) now follows from Theorem \ref{thcomp}.
In the case $p'(t)\!\leqslant \!0$ we follow the same argument as in the previous paragraph.

6) Suppose that the first inequality (\ref{dop21}) is satisfied. Let us consider the differential operator
\begin{equation}\label{dop23}
L_2x:= x''+p(t)x'+\left(\frac{p^2(t)}{4}+r(t)\right)\,x.
\end{equation}
We put 
$
v(t):=\exp\left(-\frac{1}{2}{\displaystyle\int_0^t} p(s)\,ds\right).$
Then $v(t)>0$ and
$$ (L_2v)(t)=\left(-\frac{1}{2}p'(t)+r(t)\right)\,e^{-\frac{1}{2}{\displaystyle\int_0^t} p(s)\,ds}\leqslant 0,\;t\in \mathbb R.
$$ 
Consequently,
$L_2\in\gT ((-\infty,\,+\infty)).$ The disconjugacy of (\ref{ord2.9}) now follows from Theorem \ref{thcomp}.

If the second inequality (\ref{dop21}) holds, we put $y(t):=x(-t)$ thus obtaining an equation
$$
y''-p(t)y'+q(t)y=0,$$ and a differential operator
$$L_2y:= y''-p(t)y'+\left(\frac{p^2(t)}{4}+r(t)\right)\,y, 
$$
for which we put
$v(t):=e^{\frac{1}{2}{\int_0^t} p(s)\,ds}.$ The rest of the argument is the same as above.

Suppose that the first inequality (\ref{dop22}) holds. We define $v(t):=e^{-\int_0^t p(s)\,ds}\;(>0).$ Then
$$(L_2v)(t)=\left(p^2(t)-\frac{1}{2}p'(t)-p^2(t)+\frac{p^2(t)}{4}+r(t)\right)\,e^{-\frac{1}{2}{\displaystyle\int_0^t} p(s)\,ds}\leqslant 0,\;t\in \mathbb R.$$ 
Therefore, we have $L_2\in\gT ((-\infty,\,+\infty))$, so it suffices to apply Theorem \ref{thcomp} to complete the proof. In the case the second inequality (\ref{dop22}) is satisfied, the argument is the same as above.
\end{proof}

We note that inequality (\ref{dop22}) is satisfied (in fact, as an identity), e.g., in the case when
$$
r(t)\equiv -R^2, \qquad p(t)=\dfrac{R\left(1-c^2e^{\frac{Rt}{2}}\right)}{1+c^2e^{\frac{Rt}{2}}}.
$$ 
Then equation $L_2x=0$ has solution
$$
x(t)=e^{-\displaystyle\int_0^t \dfrac{R\left(1-c^2e^{\frac{Rs}{2}}\right)}{1+c^2e^{\frac{Rs}{2}}}\,ds}\;>0.
$$

\textbf{2.~}Finally, let us consider equation
\begin{equation}\label{ord2.13}
Lx:=  x''+p(t)x'+q(t)x=0 \quad (t\in (a,+\infty))
\end{equation}
with coefficients continuous on $(a,+\infty)$. The substitution $t\to a+t^2$ transforms equation (\ref{ord2.11}) into equation
\begin{equation}\label{ord2.14}
Lx:=  x''+p(a+t^2)x'+q(a+t^2)x
=0 \quad (t\in (-\infty,+\infty)).
\end{equation}
Now, the disconjugacy of equation (\ref{ord2.12}) on $\mathbb R$ is equivalent to the disconjugacy of equation (\ref{ord2.11}) on $(a,+\infty).$ Hence, by applying the criteria for disconjugacy for equation (\ref{ord2.12}) derived in the previous sections, we obtain the respective criteria for disconjugacy of equation (\ref{ord2.11}) on $(a,+\infty).$

\vspace*{2mm}

\textbf{3.~}We now apply Theorem \ref{osnov} to the problem of existence of periodic solutions of equation (\ref{ord2.9}) (cf. \cite{kt68}--\cite{kt74}) 

The absence of a non-trivial periodic solution of a linear homogeneous equation of second order is equivalent to the existence of unique solution to periodic boundary value problem
\begin{gather}\label{ex1}
(Lx)(t):= x''+p(t)x'+q(t)x=f(t),
\end{gather}
 $(p(t+T)=p(t),\;q(t+T)=q(t),\;f(t+T)=f(t),\;T>0),$
\begin{gather}\label{ex2}
 x(a)=x(a+T),\quad x'(a)=x'(a+T)
\end{gather}
for all $a\in \mathbb R$ and any right-hand side $f.$

N.N.~Yuberev was the first who observed the relationship between the existence of unique solution to the boundary value problem
(\ref{ex1}),\,(\ref{ex2}), and the disconjugacy of equation $Lx=0$ on interval $[a,a+T]$ for any
$a\in \mathbb R$ (see \cite{juber,juber68}).
The next theorem can be easily derived from \cite{juber,juber68}.
We give a new proof of this result.

\begin{theorem}\label{teor-1}
{\it Let $q(t)\not\equiv 0,\;q(t)\geqslant 0\quad \bigl(q(t)\leqslant 0\bigr),\;t\in \mathbb R,$. Suppose that functions $p,q$ are integrable and $T$-periodic, and equation
$Lx=0$ is disconjugate on $\mathbb R.$ Then the boundary value problem $ (\ref{ex1}),\,(\ref{ex2})$ is uniquely solvable for any right-hand side $f$ or, equivalently, the homogeneous equation $Lx=0$ does not have non-trivial $T$-periodic solutions.}
\end{theorem}

We note that the conditions of Theorem \ref{teor-1} are more restrictive than the conditions of Yuberev-Tonkov-Hohryakov theorem, since the disconjugacy of the second order linear differential equation on $\mathbb R$ implies its disconjugacy on any interval $[a,a+T],$ while the converse is not true.
However, the conditions of Theorem \ref{teor-1} can be verified easily, which justifies these restrictions.

We note also that the condition of the constant sign for the coefficient
 $q$ can not be omitted. 
For instance, equation (\ref{ord2.11}) having $2\pi$-periodic coefficients is disconjugate on $\mathbb R$ (the latter follows from Theorem \ref{VP'} if we put 
$v(t)\equiv 2+\sin\,t>0\;(t\in\mathbb R)$), however it has a  $2\pi$-periodic solution $u(t)\equiv v(t)\equiv 2+\sin\,t.$

\begin{proof}[Proof of Theorem \ref{teor-1}]
Let us suppose that equation $Lx=0$ has a $T$-periodic solution $u.$ 
Since this solution can not have zeros due to the disconjugacy of the latter equation, we may assume without loss of generality that
$u(t)>0\;(t\in\mathbb R).$ Let $\{t_k\}_{k=1}^\infty,\;0<t_1<t_2<\ldots$ be the sequence of points of global minima 
of function $u$, let $m:= u(t_k)>0.$ 
We note that on each interval of length $T$ there can be only finitely many of such points, as follows from the finiteness of the total variation of function $u$ being a continuously differentiable function.
We define $z(t):= u(t)-m.$ Then
\begin{gather}\label{ex5}
z(t)\geqslant 0\quad (t\in\mathbb R),\quad z(t_k)=0,\quad k=1,2,\dots
\end{gather}
$(Lz)(t)=(Lu)(t)-mq(t))=-mq(t)$ and function $z,$ being a solution of  boundary value problem
\begin{gather}\label{ex6}
(Lz)(t)=-mq(t),\quad t\in [t_1,t_n],\quad u(t_1)=0,\;u(t_n)=0
\end{gather}
has presentation (we assume that $n>1$ is sufficiently large, so that between each two points $t_1$ and $t_n$ on distinct period intervals there would exist at least one point $t_k\in (t_1,t_n)$) 
$$
z(t)=m\int\limits_{t_1}^{t_n}\Bigl(-G_n(t,s)\Bigr)q(s)\,ds>0\quad (<0)\quad (t\in (t_1,t_n)),
$$
where $G_n$ is the Green function of problem (\ref{ex6}). As is well known (see, e.g., \cite{lev69})  $G_n(t,s)<0$ for $(t,s)\in (t_1,t_n)^2.$
The inequality $z(t)>0$ $(z(t)<0$) for $t\in (t_1,t_n)$ contradicts to (\ref{ex5}). This completes the proof.
\end{proof}

Theorems \ref{osnov} and \ref{teor-1} now give us the following result.

\begin{corollary}
{\it Suppose that $q(t)\not\equiv 0,\;q(t)\geqslant 0\quad \bigl(q(t)\leqslant 0\bigr),\;t\in \mathbb R$, where $p,q$ are locally integrable $T$-periodic functions,
and equation $Lx=0$ satisfies one of conditions $1)$ -- $6)$ of Theorem \ref{osnov}. Then the homogeneous equation $Lx=0$ does not have non-trivial $T$-periodic solutions.}
\end{corollary}

\end{document}